%% file: submited.20030909.tex
\documentclass[a4paper,12pt]{article}

\usepackage{amsmath}
\usepackage{amssymb}
\usepackage{theorem} 
\usepackage{rotating}

\newcommand{\qed}{\relax\ifmmode\hskip2em\Box\else\unskip\nobreak\hfill$\Box$\fi}

\newcommand{\N}{\ensuremath{\mathbb{N}}}

\newcommand{\comb}[2]{\ensuremath{{#1 \choose #2}}}

\newtheorem{defeng}{Definition}

\newtheorem{theorem}{Theorem}

\newtheorem{lemma}{Lemma}
\newtheorem{proposition}{Proposition}

{\theorembodyfont{\rmfamily} }
{\theorembodyfont{\rmfamily} }
{\theorembodyfont{\rmfamily} }
{\theorembodyfont{\rmfamily} }
{\theorembodyfont{\rmfamily} }
	{\noindent {\bf Preuve}{#1}{\bf\ : }}{\qed\vspace{2ex}}
\newenvironment{proof}[1][]%
	{\noindent {\bf Proof}{#1}{\bf: }}{\qed\vspace{2ex}}

\title{Ramsey-type results  on singletons, co-singletons  and monotone
sequences in large collections of sets}

\author{Sylvain Gravier\thanks{C.N.R.S., Laboratoire Leibniz --- IMAG,
     46    Avenue    F\'elix    Viallet,   38031    Grenoble    Cedex.
     Sylvain.Gravier@imag.fr;      Frederic.Maffray@imag.fr}      \and
     \addtocounter{footnote}{-1}   Fr\'ed\'eric   Maffray\footnotemark
     \and   J\'er\^ome  Renault\thanks{CEREMADE,   Universit\'e  Paris
     Dauphine.      renault@ceremade.dauphine.fr}     \and     Nicolas
     Trotignon\thanks{Laboratoire Leibniz ---  IMAG, 46 Avenue F\'elix
     Viallet,        38031       Grenoble        Cedex.       \newline
     Nicolas.Trotignon@imag.fr}}

\date{September 9, 2003}

\begin{document}

\bibliographystyle{plain}

\maketitle

\begin{abstract}

We say  that a 0-1 matrix  $N$ of size $a\times  b$ can be  found in a
  collection of sets $\mathcal{H}$ if  we can find sets $H_{1}, H_{2},
  \dots, H_{a}$  in $\mathcal{H}$ and elements $e_1,  e_2, \dots, e_b$
  in  $\cup_{H \in  \mathcal{H}} H$  such  that $N$  is the  incidence
  matrix of  the sets $H_{1},  H_{2}, \dots, H_{a}$ over  the elements
  $e_1, e_2, \dots, e_b$.   We prove the following Ramsey-type result:
  for every $n\in  \N$, there exists a number $S(n)$  such that in any
  collection  of  at  least  $S(n)$  sets, one  can  find  either  the
  incidence  matrix  of  a   collection  of  $n$  singletons,  or  its
  complementary matrix, or the incidence matrix of a collection of $n$
  sets completely ordered by inclusion. We give several results of the
  same extremal set  theoretical flavour.  For some of  these, we give
  the exact value of the number of sets required.

\vspace{2ex}

\noindent {\bf R\'esum\'e\ (in french):}

On dit qu'une matrice 0-1 $N$ de taille $a\times b$ se trouve dans une
  collection  d'ensembles  $\mathcal{H}$  si  l'on  peut  trouver  des
  ensembles  $H_{1}, H_{2},  \dots, H_{a}$  dans $\mathcal{H}$  et des
  \'el\'ements $e_1,  e_2, \dots, e_b$ dans  $\cup_{H \in \mathcal{H}}
  H$ tels que $N$ soit la  matrice d'incidence de la trace des $H_{1},
  H_{2}, \dots,  H_{a}$ sur les  \'el\'ements $e_1, e_2,  \dots, e_b$.
  Nous d\'emontrons  le r\'esultat suivant de type  Ramsey~: Pour tout
  $n\in \N$, il existe un  nombre $S(n)$ tel que dans toute collection
  d'au  moins $S(n)$ ensembles  distincts, on  trouve soit  la matrice
  d'incidence   d'une   collection   de   $n$  singletons,   soit   le
  compl\'ementaire de cette matrice, soit la matrice d'incidence d'une
  collection de  $n$ ensembles totalement  ordonn\'es par l'inclusion.
  Nous   donnons   quelques   r\'esultats  similaires   de   th\'eorie
  extr\`emale des ensembles.  Pour  certains d'entre eux, nous donnons
  le nombre exact d'ensembles requis.

\end{abstract}

\section*{Introduction}
We give  here results  about the local  structure of any  large enough
  family of distinct sets.  By  local structure, we mean a description
  of some  of the sets for some  of the elements of  their union.  Our
  Lemma~\ref{lemme1}  states that  in any  large enough  collection of
  distinct sets,  one can find  an ``increasing'' or  a ``decreasing''
  sequence, in a weak sense described below.

Bauslaugh~\cite{bauslaugh:list} originally gave an infinite version of
   that lemma and used it to find in any infinite twinless digraph some
   special  induced  subdigraph, thus  giving  a  counter-example to  a
   property   of  compactness  for   list-colouring.   But   the  proof
   in~\cite{bauslaugh:list} has an error%
\footnote{In~\cite{bauslaugh:list},  in the  proof of  Lemma 7,  it is
claimed (page 21,  line 18) that ``$|S \backslash  (S \cap A_i)|$ must
take on arbitrary large  finite values \dots''.  However, the instance
$A_i=\N \backslash \{i\}$ with $A=S=\N$ satisfies all the requirements
while $|S \backslash (S \cap A_i)|$ takes only value 1.},
and our proof (section~\ref{sectionL}) may be considered as an erratum
   to~\cite{bauslaugh:list}.

Using Lemma~\ref{lemme1} and Ramsey theory, we prove that in any large
     enough  family  of  distinct  sets,  we  can  find  a  very  precise
     substructure    (Theorem~\ref{lemmealaR},   section~\ref{sectionR}).
     While we cannot give an exact bound in general, we provide lower and
     upper bounds and we give an exact bound for special cases.

F\"{u}redi and  Tuza \cite{furedi:hypergraphs} gave a  theorem that is
    more  precise than Lemma~\ref{lemme1}  in the  case where  the sets
    under  consideration   are  ``small''.   We  then   use  the  ideas
    introduced  in section~\ref{sectionR} and  the ideas  of F\"{u}redi
    and Tuza to get a new result (section \ref{sectionF}).\\

For any non-negative integer $n$,  we denote by $[n]$ the set $\{1, 2,
    \dots,  n\}$ ($[0]=\emptyset$).   If  $E$ is  a  finite set,  $|E|$
    denotes its cardinality and for each $k$, $\comb{E}{k}$ denotes the
    collection of all subsets of $E$ of size $k$.

\section{Increasing and decreasing sequence in large enough collections
     of sets}
\label{sectionL}

We first define increasing and  decreasing sequences of sets.

\begin{defeng}
Let $(H_1, H_2, \dots, H_k)$ be a  finite sequence of sets.

\begin{itemize}

\item
     The sequence is \emph{$(k-1)$-increasing} if $H_{1} \subsetneq
H_{1}  \cup  H_{2}
     \subsetneq  \dots   \subsetneq  H_{1}  \cup   H_{2}\cup  \dots  \cup
     H_{{k}}$.

Similarly, a  countably infinite sequence of sets  $(H_1, H_2, \dots)$
     is  \emph{increasing} if  and only  if $H_{1}  \subsetneq H_{1}  \cup H_{2}
     \subsetneq \dots $

\item
The sequence is \emph{$(k-1)$-decreasing} if $H_{1} \supsetneq  H_{1}
\cap  H_{2}
     \supsetneq  \dots   \supsetneq  H_{1}  \cap   H_{2}\cap  \dots  \cap
     H_{{k}}$.

Similarly, a countably infinite  sequence of sets $(H_1, H_2, \dots)$
     is \emph{decreasing} if and only if $H_{1} \supsetneq H_{1} \cap
     H_{2} \supsetneq \dots $
\end{itemize}
\end{defeng}

How large should be a collection of sets to contain  a $(k-1)$
increasing or decreasing sequence? This is answered by our
first lemma.

\begin{lemma} \label{lemme1}
Let  $k$, $l$  and $m$  be  in $\N$, and $\{H_1,H_2,  \dots,
     H_m\}$ be a collection of $m$ distinct sets.  If $m > \comb{k+l}{l}$
     then at least one of the following two statements holds:

\begin{itemize}

\item Among the $H_i$'s, one can find a $(k+1)$-increasing sequence of
sets $(H_{i_1}, H_{i_2}, \dots, H_{i_{k+2}})$.

\item  Among  the  $H_i$'s,  one  can  find  an $(l+1)$-decreasing
      sequence of sets $(H_{i_1},  H_{i_2},  \dots,
H_{i_{l+2}})$.

\end{itemize}

\end{lemma}

\begin{proof}

We proceed  by induction on $k, l$.   If $k=0$ or $l=0$,  the lemma is
     clear.  Assume now  $k>0$ and $l>0$, and let $\{H_1,...,H_m\}$ be a
collection of $m$ distinct sets with  $m >
\comb{k+l}{l}$.  So   $m\geq  2$, and there exists
$x$ in  $(H_1\cup  H_2\cup \dots \cup  H_m )\backslash (H_1\cap
H_2\cap \dots \cap  H_m )$.
Let  $m_1$ (resp. $m_2$) be the
number of
     sets among  $H_1, H_2,  \dots , H_m$ that  contain $x$  (resp. that do
     not contain $x$).   So $m_1$ and $m_2$ are positive and  $m=m_1+m_2$.
Since ${k+l  \choose  l} =
     {k+l-1 \choose l-1}+{k-1+l \choose l}$, at least one of the two
following cases holds:

\begin{itemize}

\item $m_1 > {k+l-1 \choose l-1}$. By the induction hypothesis we find
among the  sets that contain  $x$ a $(k+1)$-increasing sequence  or an
$l$-decreasing sequence. In the first  case we are done. In the second
one it  suffices to append any  set without $x$  to the $l$-decreasing
sequence to get an $(l+1)$-decreasing sequence.

\item   $m_2  >   {k-1+l   \choose  l}$.    Similarly,   we  find   an
$(l+1)$-decreasing  sequence,  or  a  $(k+1)$-increasing  sequence  by
appending  any set  with  $x$  to a  $k$-increasing  sequence of  sets
without $x$.

\end{itemize}
\vbox{}
\end{proof}

The tightness  of the bound  $\comb{k+l}{l}$ in Lemma~\ref{lemme1} is
      established    by    considering    the    collection    of    sets
      $\comb{[k+l]}{l}$.

In  any  infinite  collection  of   distinct  sets,  we  can  find  by
     Lemma~\ref{lemme1}  an arbitrarily  long increasing  or decreasing
     sequence.  But  this does not  immediately imply that there  is an
     infinite increasing or decreasing sequence.  This is why we recall
     and   prove  here   the  infinite   lemma  originally   stated  by
     Bauslaugh~\cite{bauslaugh:list}.    One  could   try  to   find  a
     compactness argument (see~\cite{nicolas:dea})  to establish a link
     between  the finite  lemma (Lemma~\ref{lemme1})  and  the infinite
     lemma below:

\begin{lemma}  \label{lemmeinf}
Let  $\mathcal{H}=\{H_1,H_2, \dots  \}$ be  an infinite  collection of
    distinct sets.  One of the two following propositions holds:

\begin{itemize}
\item Among the $H_i$'s, one  can find an infinite increasing sequence
$(H_{i_1}, H_{i_2}, \dots)$.

\item Among the $H_i$'s, one  can find an infinite decreasing sequence
$(H_{i_1}, H_{i_2}, \dots)$.
\end{itemize}

\end{lemma}

\begin{proof}
We  claim  that  there  exists an  infinite  sequence  $(x_1,H_{k_1}),
    (x_2,H_{k_2}), \dots$,  such that  for every $i  \geq 1$ one  of the
    following two properties holds:

\begin{enumerate}

\item $x_i\notin H_{k_i}$ and for every $j>i$, $x_i\in H_{k_j}$.

\item $x_i\in H_{k_i}$ and for every $j>i$, $x_i\notin H_{k_j}$.

\end{enumerate}

We establish the claim by induction on $i$.  For $i=1$, pick any $x_1$
    which lies in at least one $H_k$ but not in all of them.

If $x_1$ lies  in infinitely many $H_k$'s, then, let $H_{k_1}$ be one
     $H_k$  that does not  contain $x_1$.   Continue with  the (infinite)
     collection of all $H_k$'s that contain $x_1$.  If $x_1$ lies in only
     finitely many $H_k$'s, then, let $H_{k_1}$ be one of them.  Continue
     with the  (infinite) collection of  all $H_k$'s that do  not contain
     $x_1$.  The  proof is entirely similar  for each $i\geq  1$.  So the
     claim is proved.

Now,  one the two  properties 1,  2, holds  for infinitely  many pairs
     $(x_i,  H_{k_i})$.  If  it is  property  1, we  find an  increasing
     sequence, and if it is property 2, we find a decreasing sequence.
\end{proof}

Note that in  Lemma~\ref{lemme1} very little is required  of the sets:
    they do not have  to be subsets of a given set, or  to be of a given
    size, or even to be finite.   But the lemma does not tell much about
    the structure one may hope to find in a sufficiently large family of
    distinct sets,  and one may suspect  that a better  result is hidden
    behind  our   lemma.   Before  going  further,   we  introduce  some
    definitions.

It  will  be convenient  to  work  with  incidence matrices.  For  any
    collection of  sets $\mathcal{H}$, and  any 0-1 matrix $N$  with $a$
    rows and $b$ columns, we say  that $N$ can be found in $\mathcal{H}$
    if  we  can find  distinct  sets  $H_{1},  H_{2}, \dots,  H_{a}$  in
    $\mathcal{H}$  and  distinct  elements  $e_1, e_2,  \dots,  e_b$  in
    $\cup_{H \in \mathcal{H}}  H$ such that $N$ is  the incidence matrix
    of the  sets $H_{1},  H_{2}, \dots, H_{a}$  over the  elements $e_1,
    e_2, \dots, e_b$ (ie $N_{\alpha,\beta}=1$ if and only if $e_\beta\in
    H_{\alpha}$).

We say that  a 0-1 matrix $N$ is a  $k$-\emph{increasing} matrix if it
    has $k$  columns, $k+1$ rows and satisfies:  $N_{i+1,i}=1$ for every
    $i\in [k]$ and  $N_{i,j}=0$ for every $1 \leq i \leq  j \leq k$.  We
    say  that  $N$ is  a  $k$-\emph{decreasing}  matrix  if it  has  $k$
    columns, $k+1$ rows and  satisfies: $N_{i,i}=0$ for every $i\in [k]$
    and $N_{i,j}=1$ for every $1 \leq j < i \leq k$.

\begin{figure}[h]
\[
\begin{array}{ccc}

\begin{array}{c}
\left(
\begin{array}{cccc}
0 	& \cdots 	&\cdots	& 0 	 \\
1 	& \ddots 	&       & \vdots \\
?      	& \ddots 	& \ddots& \vdots \\
\vdots 	& \ddots	& \ddots& 0 	\\
?	& \cdots	& ?     & 1
\end{array}
\right)
\\
\text{Increasing}
\end{array}

&
\begin{array}{c}
\left(
\begin{array}{cccc}
0 	& ?      	&\cdots	& ? 	 \\
1 	& \ddots 	& \ddots& \vdots \\
\vdots 	& \ddots 	& \ddots& ?      \\
\vdots 	&       	& \ddots& 0 	 \\
1	& \cdots	& \cdots& 1
\end{array}
\right)
\\
\text{Decreasing}
\end{array}

\end{array}
\]

\caption{Increasing and decreasing matrices}\label{fig:id}
\end{figure}

Lemma~\ref{lemme1}    can    be     rephrased    as    follows:    Let
    $\mathcal{H}=\{H_1,H_2,  \dots,  H_m\}$   be  a  collection  of  $m$
    distinct  sets.   If  $m  >  \comb{k+l}{l}$ then  one  can  find  in
    $\mathcal{H}$ a  $(k+1)$-increasing matrix or  an $(l+1)$-decreasing
    matrix.

\section{Finding more specific matrices}
\label{sectionR}

As  noted by  Bauslaugh in  his study  of infinite  digraphs, Ramsey's
    famous theorem  may be  combined with Lemma~\ref{lemmeinf}.   In our
    finite  extremal set-theoretic  context, this  gives a  more precise
    idea of the  kind of local structure that can be  found in any large
    enough collection of sets.

For any integer $n\geq 1$  we call $n$-\emph{singleton} matrix the 0-1
     matrix  $S^n$   with  $n$  columns   and  $n+1$  rows   defined  by
     $S^n_{i,j}=1$    if    and    only    if    $i=j+1$.     We    call
     $n$-\emph{co-singleton} matrix the  0-1 matrix $\bar{S}^n$ with $n$
     columns and  $n+1$ rows defined by $\bar{S}^n_{i,j}=0$  if and only
     if $i=j$.  We call  $n$-\emph{monotone} matrix the 0-1 matrix $M^n$
     with $n$ colomns and $n+1$ rows defined by $M^n_{i,j}=1$ iff $i\geq
     j+1$.

\begin{figure} [h]
\[
\begin{array}{ccc}

\begin{array}{c}
\left(
\begin{array}{cccc}
0 	& \cdots 	&\cdots	& 0 	 \\
1 	& 0	 	&\cdots	& 0      \\
0 	& \ddots 	& \ddots& \vdots \\
\vdots 	& \ddots	& \ddots& 0 	\\
0	& \cdots	& 0	& 1
\end{array}
\right)
     \\
\text{Singleton}
\end{array}

&
\begin{array}{c}
\left(
\begin{array}{cccc}
0 	& 1      	&\cdots	& 1      \\
1 	& \ddots 	& \ddots& \vdots \\
\vdots 	& \ddots	& \ddots& 1 	\\
1	& \cdots	& 1	& 0	\\
1 	& \cdots 	&\cdots	& 1
\end{array}
\right)
\\
\text{Co-singleton}
\end{array}

&
\begin{array}{c}
\left(
\begin{array}{cccc}
0 	& \cdots 	&\cdots	& 0 	 \\
1 	& \ddots 	& 	& \vdots \\
\vdots 	& \ddots 	& \ddots& \vdots \\
\vdots 	& 		& \ddots& 0 	\\
1	& \cdots	& \cdots& 1
\end{array}
\right)
\\
\text{Monotone}
\end{array}

\end{array}
\]

\caption{Singleton, co-singleton and monotone matrices}\label{fig:scm}
\end{figure}

Notice  that  $S^1  =  M^1   =  \bar{S}^1$.   If  $n>1$,  then  $S^n$,
    $\bar{S}^n$  and  $M^n$  are  distinct. Every  singleton  matrix  is
    increasing  and  every   co-singleton  matrix  is  decreasing.   The
    matrices which  are both increasing and decreasing  are the monotone
    matrices.  We  call \emph{complementary} of a matrix  $N$ the matrix
    obtained from $N$ by swapping  $0$ and $1$.  Up to rearrangements of
    the  rows and/or  the columns,  the complementary  of  a cosingleton
    matrix is  a singleton  matrix and the  complementary of  a monotone
    matrix is a monotone matrix.

    We are going to find an appropriate singleton, cosingleton or
monotone matrix in any large enough
collection of sets. We
first recall Ramsey's theroem.

\begin{theorem}[Ramsey, see \cite{grh:ramseytheory}]

For any positive integer $r$  there exists a positive integer $n$ such
    that for  every partition $(A_0, A_1)$  of $[n] \choose  2$, one can
    find a subset $A'$ of $[n]$ such that: $\left(|A'|\geq r\right)$ and
    either ${A'\choose2}\subseteq A_0$ or ${A'\choose2}\subseteq A_1$.

\end{theorem}

We denote by $R(r)$ the  Ramsey number, i.e., the smallest integer $n$
     that  satisfies the  claim of  the Ramsey  theorem  (for instance,
     $R(3)=6$).  The  exact value  of $R(r)$ is  not known  in general,
     even for small values of $r$, although some lower and upper bounds
     have been given (see~\cite{grh:ramseytheory}).

\begin{theorem} \label{lemmealaR}
For every non-negative integers $k$ and $l$, there exists a number $S$
    such that for any  collection of sets $\mathcal{H}$ , $|\mathcal{H}|
    > S$  implies that  at least  one of the  following  three propositions
    holds:
\begin{itemize}

\item The $(k+1)$-singleton matrix can be found in $\mathcal{H}$.
\item The $(l+1)$-cosingleton matrix can be found in $\mathcal{H}$.
\item   The  $\min   (k+1,l+1)$-monotone  matrix   can  be   found  in
$\mathcal{H}$.

\end{itemize}

\noindent We denote by   $S(k,l)$ the largest  integer that
does not satisfy the  claim.  We
      have:
$$ S(k,l)=S(l,k)\leq   \comb{R(k+1)+R(l+1)-2}{R(k+1)-1}.$$

\end{theorem}

\begin{proof}
Let  $k$  and   $l$  be  in $\N$,  and   consider  a  collection
     $\mathcal{H}$   of  distinct   sets  such   that   $|\mathcal{H}|  >
     \comb{R(k+1)  + R(l+1)-2}{R(k+1)  - 1}$.  By  Lemma~\ref{lemme1}, we
     find  in  $\mathcal{H}$  an  $R(k+1)$-increasing matrix  $N$  or  an
     $R(l+1)$-decreasing matrix $N'$.

In  the  first  case,  let   $A_0$  (resp.~$A_1$)  be  the  subset  of
     $[R(k+1)]\choose 2$ consisting of the $\{i,j\}$'s such that $i>j$
     and $N_{i+1,j}=0$ (resp.~$N_{i+1,j}=1$).  By Ramsey's theorem, we
     can find a subset of  $[R(k+1)]$, say without loss of generality,
     the subset  $[k+1]$ such that  all the pairs in  $[k+1]\choose 2$
     are in $A_0$ or in $A_1$.  If they are in $A_0$, we have found in
     $\mathcal{H}$ a $(k+1)$-singleton matrix.   If they are in $A_1$,
     we  have found  a  $(k+1)$-monotone matrix.  The  second case  is
     similar.

Thus      $S$      exists       and      we      have      $S(k,l)\leq
    \comb{R(k+1)+R(l+1)-2}{R(k+1)-1}.$  The  claim  $S(k,l)=S(l,k)$  is
    clear by complementation.

\end{proof}

Note that an analogue of  Theorem~\ref{lemmealaR} with only two of the
     three cases considered would be  false.  To see this, it suffices to
     consider  the situation  when $\cal  H$ itself  is a  collections of
     singletons,  or a collection  of co-singletons,  or a  collection of
     sets completely ordered by inclusion.

\subsection{Some exact values for  $S(k,l)$}

Since the exact value of the Ramsey number is not known in general, it
     could  seem  hopeless  to   try  to  determine  $S(k,l)$  exactly.
     Nevertheless, for  small values  of $k$ and  $l$, we can  give the
     exact  value of  $S(k,l)$.  It  appears that  the upper  bound for
     $S(k,l)$ using the Ramsey  number is quite generous (for instance,
     it says $S(2,2) \leq {10 \choose 5} = 252$).

The collection  $\comb{[k+l]}{l}$ shows $S(k,l)\geq  {k+l \choose l}$.
     Actually, for  $l=0$, we do have  $S(k,l) = {k+l \choose  l} = 1$
     This simply says that if at least two distinct sets are given, the
     matrix $0 \choose 1$ can be found in them.

If $l=1$, the situation is also simple:

\begin{lemma}
If $l=1$, $S(k,l) = {k+l \choose l} = k+1$ for every $k$ in $\N$.
     \end{lemma}

\begin{proof}
The proof is easy by a  direct induction on $k$.  We give here another
     proof:   Let  $\mathcal{H}$   be  a   collection  of   sets.   If
     $|\mathcal{H}|  > k+1$,  we  want to  find  in $\mathcal{H}$  the
     $(k+1)$-singleton  matrix,   the  2-cosingleton  matrix   or  the
     2-monotone   matrix.    By    Lemma~\ref{lemme1}   we   find   in
     $\mathcal{H}$   a  $(k+1)$-increasing  matrix   (case  1)   or  a
     2-decreasing  matrix  (case 2).   In  case  1,  if by  fluke  the
     $(k+1)$-increasing matrix is  the $(k+1)$-singleton matrix we are
     done.  If  not, we find  in $\mathcal{H}$ the  2-monotone matrix.
     In case 2, we are done  since a 2-decreasing matrix is either the
     2-cosingleton matrix or the 2-monotone matrix.
\end{proof}

As it is  true for $l=0$ and $l=1$, one  could think that $S(k,l)={k+l
     \choose l}$ in  general.  But this is false  for $k=l=2$: the matrix
     $F$ below shows that $S(2,2)\geq 8$.  Indeed, $F$ should be seen as the
     incidence  matrix of eight  distinct sets  over four  elements.  The
     point is  that $S^3$, $\bar{S}^3$ or $M^3$  are not submatrices
     of $F$ even after rearranging the rows and the columns.

\[
F= \left(
\begin{array}{cccc}
1000\\
0100\\
1101\\
1110\\
0011\\
0110\\
1001\\
1100
\end{array} \right)
\]

Actually, we proved  that $S(2,2)=8$. The proof  is long, and
all the details are in the appendix.

\newcommand{\je}[3]{\ensuremath{{c_#1}\rightarrow[r_#2 \simeq r_#3]}}

\subsection{A lower  bound for $S(l,l)$}

The exact value  of $S(k,l)$ in general seems  difficult to determine.
     We already noted that $S(k,l) \geq {k+l \choose l}$. Better bounds
     can be found.

\begin{proposition} \label{pro2}
For $l\geq 2$,   $S(l,l) \geq {2l \choose l} +
     {2l-3 \choose l-1}$.
\end{proposition}

\begin{proof}
    We consider the collection ${\cal H} = A_0 \cup
     A_1  \cup   \dots  \cup  A_{l-2}  \cup   B  \cup  C   \cup  D$, with:

\begin{eqnarray*}
A_i & = & \left\{ H\in {[2l] \choose i+1} \text{ s.t. } 2l \in H \text{ and }
H\backslash\{2l\} \in {[l+i-1] \choose i} \right\} \\
B & = & \left\{ H\in {[2l] \choose l} \text{ s.t. } 1 \in H \text{ and }
2l-1 \notin H \text{ and }
2l \in H \right\} \\
C & = & {[2l-1] \choose l} \\
D & = & \left\{ H\in {[2l] \choose l+1} \text{ s.t. } 1 \in H \text{ and }
2l \in H \right\}
\end{eqnarray*}

The transpose  of the  incidence matrix of  ${\cal H}$ for  $l=3$ is
    given in figure~\ref{fig23}.

We  have  $\displaystyle  |{\cal  H}|=\left(  \sum_{i=0}^{l-2}  {l+i-1
    \choose i} \right)+  { 2l-3 \choose l-2} +{ 2l-1  \choose l} + {2l-2
    \choose  l-1}$.  Since $  \sum_{i=0}^{l-2}  {l+i-1 \choose  i}={2l-2
    \choose l}$,  we obtain that $|{\cal  H}|= { 2l-3  \choose l-1} +{2l
    \choose  l}$.  We  now  prove  that  $S^{l+1}$,  $\bar{S}^{l+1}$  or
    $M^{l+1}$ cannot be found in ${\cal H}$.

Assume that we can find $M^{l+1}$ in ${\cal H}$. Then we can find sets
    $H_1$  and   $H_2$  in  ${\cal   H}$  and  an   increasing  sequence
    $c_1$,...,$c_{l+1}$  in  $[2l]$ s.t.  for  each $k=1,...,l+1$,  $c_k
    \notin H_1$  and $c_k \in H_2$.  $|H_2|\geq l+1$ gives  $H_2 \in D$,
    and   $c_{l+1}=2l$.   But   $|H_1|   \leq  l-1$   gives   $H_1   \in
    \cup_{i=0}^{l-2}  A_i$, and  we  have a  contradiction since  $2l\in
    H_1$.

Assume now that  we can find $\bar{S}^{l+1}$ in  ${\cal H}$. Denote by
    $H$ in ${\cal H}$ the set  corresponding to the rows with all 1's in
    $\bar{S}^{l+1}$,  and  by  $c_1<c_2<...<  c_{l+1}$ the  elements  in
    $[2l]$ corresponding to these 1's.   We have $|H|\geq l+1$, hence $H
    \in  D$,  $c_1=1$ and  $c_{l+1}=2l$.  We then  have  a  set $H'$  in
    $\cal{H}$ s.t. $|H'|\geq l$, 1 $\notin H'$ and $2l \in H'$. But such
    a set does not exist.

Finally assume  that we can  find $S^{l+1}$ in $\cal{H}$.   Then there
    exist  sets  $H$, $H_1$,  ...,  $H_{l+1}$  in $\cal{H}$,  elements
    $c_1<c_2<...<c_{l+1}$    in    $[2l]$    such   that    $H    \cap
    \{c_1,...,c_{l+1}\}=  \emptyset$ and  for $j$  in $\{1,...,l+1\}$,
    $H_j \cap  \{c_1,...,c_{l+1}\}= \{c_j\}$.  We  have $|H|\leq l-1$,
    thus  $H\in \cup_{i=0}^{l-2}A_i$,  $2l \in  H$ and  $c_{l+1} <2l$.
    For each $j=1,...,l+1$, $H_j$ has  at least $l$ 0's in the columns
    $1,...,2l-1$, hence $H_j  \notin D$ and $H_j \notin  C$.  Since no
    set in  $\cup_{i=0}^{l-2}A_i \cup  B$ contains $2l-1$,  this imply
    that $c_{l+1} \neq 2l-1$ and  for each $j=1,...,l+1$, $H_j$ has at
    least $l$ 0's  in the columns $1,...,2l-2$, hence  $H_j \notin B$.
    We  have obtained  that  for each  $j$  in $\{1,...,l+1\}$,  there
    exists a  (necessarily unique) $i_j$ in  $\{0,...,l-2\}$ such that
    $H_j \in  A_{i_j}$.  Fix  $j$ with maximum  index $i_j$.   We have
    $c_{l+1} \leq l+i_j  -1 \leq 2l-3$.  Hence $H_j$  has at least $l$
    0's in the columns $1,...,l+i_j-1$. But $H_j \backslash \{2l\} \in
    {[l+i_j-1] \choose i_j}$, hence a contradiction.

\newcommand{\ro}[1]{\begin{rotate}{0}{#1}\end{rotate}}
{\setlength{\arraycolsep}{0.5em}
\begin{figure}
\[
\begin{array}{cccccc}
\begin{array}{c}
\begin{array}{c}
\ro6\\
\ro5\\
\ro4\\
\ro3\\
\ro2\\
\ro1
\end{array}
\\
\mbox{}
\end{array}&
\begin{array}{c}
\begin{array}{c}
\ro1\\
\ro0\\
\ro0\\
\ro0\\
\ro0\\
\ro0
\end{array}
\\
A_0
\end{array}&
\begin{array}{c}
\begin{array}{ccc}
\ro1&\ro1&\ro1\\
\ro0&\ro0&\ro0\\
\ro0&\ro0&\ro0\\
\ro1&\ro0&\ro0\\
\ro0&\ro1&\ro0\\
\ro0&\ro0&\ro1
\end{array}
\\
A_1
\end{array}&
\begin{array}{c}
\begin{array}{ccc}
\ro1&\ro1&\ro1\\
\ro0&\ro0&\ro0\\
\ro1&\ro0&\ro0\\
\ro0&\ro1&\ro0\\
\ro0&\ro0&\ro1\\
\ro1&\ro1&\ro1
\end{array}
\\
B
\end{array}&
\begin{array}{c}
\begin{array}{cccccccccc}
\ro0&\ro0&\ro0&\ro0&\ro0&\ro0&\ro0&\ro0&\ro0&\ro0\\
\ro1&\ro1&\ro1&\ro0&\ro1&\ro1&\ro0&\ro1&\ro0&\ro0\\
\ro1&\ro1&\ro0&\ro1&\ro1&\ro0&\ro1&\ro0&\ro1&\ro0\\
\ro1&\ro0&\ro1&\ro1&\ro0&\ro1&\ro1&\ro0&\ro0&\ro1\\
\ro0&\ro1&\ro1&\ro1&\ro0&\ro0&\ro0&\ro1&\ro1&\ro1\\
\ro0&\ro0&\ro0&\ro0&\ro1&\ro1&\ro1&\ro1&\ro1&\ro1
\end{array}
\\
C
\end{array}&
\begin{array}{c}
\begin{array}{cccccc}
\ro1&\ro1&\ro1&\ro1&\ro1&\ro1\\
\ro1&\ro1&\ro0&\ro1&\ro0&\ro0\\
\ro1&\ro0&\ro1&\ro0&\ro1&\ro0\\
\ro0&\ro1&\ro1&\ro0&\ro0&\ro1\\
\ro0&\ro0&\ro0&\ro1&\ro1&\ro1\\
\ro1&\ro1&\ro1&\ro1&\ro1&\ro1
\end{array}
\\
D
\end{array}
\end{array}
\]
\caption{Incidence matrix of $\cal H$}
\label{fig23}
\end{figure}
}

\end{proof}

\section{An exact bound for subsets of $[k+l]$}
\label{sectionF}

F\"uredi and Tuza  gave a theorem that, in a  sense, is an improvement
    of  Lemma~\ref{lemme1}.   It  states   that  if  all  the  sets  are
    ``small'', a very special increasing matrix (a singleton matrix) can
    be found. In what follows, $k$ and $l$ are non-negative integers.

\begin{theorem}[F\"uredi, Tuza \cite{furedi:hypergraphs}] \label{thfuredi}
Let  $\mathcal{H}$ be a  collection of  distinct sets  $H_1, H_2, \dots,
      H_m$.  If  $m >  {k+l\choose l}$  and if we  have $|H_i|\leq  l$ for
      every  $i$, then we  can find  in $\mathcal{H}$  a $(k+1)$-singleton
      matrix.
\end{theorem}

The proof  of that  theorem is based  on the following  theorem proved
     independently by Frankl and Kalai:

\begin{theorem}[Frankl \cite{frankl:two}; Kalai
\cite{kalai:intersection}] \label{thfrankl}

Let $A_1, A_2, \dots,  A_m$ be sets of size at most  $l$ and let $B_1,
     B_2,  \dots,  B_m$  be  sets  of  size at  most  $k$  with  $A_i\cap
     B_i=\emptyset$.   Suppose  that   $A_i\cap  B_j\neq  \emptyset$  for
     all $i>j$. Then $m\leq \comb{k+l}{l}$.

\end{theorem}

The ideas of F\"uredi and Tuza can  be be used to provide a new result
     which  looks like Theorem~\ref{lemmealaR} except  that here,  we do
     have an exact bound as proved by  the collection $\comb{[k+l]}{l}$:

\begin{theorem} \label{notth}
Let $\mathcal{H}$  be a collection  of distinct sets $H_1,  H_2, \dots
     H_m$, all  of them  included in $[k+l]$.   If $m  > \comb{k+l}{l}$
     then at least one of the following three conditions is true:

\begin{enumerate}

\item The $(k+1)$-singleton  matrix can be found in  the collection of
the sets of $\mathcal{H}$ that have at most $l$ elements.

\item The $(l+1)$-cosingleton matrix can be found in the collection of
the sets of $\mathcal{H}$ that have at least $l+1$ elements.

\item For some  $i\neq j$, $H_j\subset H_i$, $|H_j|\leq  l$ and $|H_i|
\geq l+1$.

\end{enumerate}

\end{theorem}

\begin{proof}

Let $\mathcal{H}$  be a collection  of distinct sets $H_1,  H_2, \dots
    H_m$, all  of them included  in $[k+l]$, and  such that none  of the
    three  conditions 1,  2, 3  hold.  We are  going to  show $m  \leq
    \comb{k+l}{l}$, thus  proving the  theorem.  If $H$  is a  subset of
    $[k+l]$, $\overline{H}$ denotes the complementary of $H$ in $[k+l]$.

Suppose  w.l.o.g.   that for  every  $i\leq  j$  we have  $|H_i|  \leq
|H_j|$. Let $n$ be the integer such that: $|H_1| \leq l, |H_2| \leq l,
\dots,   |H_n|   \leq  l,   |\overline{H_{n+1}}|   \leq  k-1,   \dots,
|\overline{H_{m}}| \leq k-1$. Note that $n$ may be $0$ or $m$.

Let:    $A_1=H_1$,   $A_2=H_2$,    \dots,   $A_n=H_n$,    $B_{n+1}   =
    \overline{H_{n+1}}$, \dots, $B_m = \overline{H_{m}}$.

For every set  $A_i=H_i$, $i\leq n$, we claim that  we can construct a
    set $B_i$ such that $A_i\cap B_i  = \emptyset$, $|B_i| = k$ and for
    every  $j$, $(1 \leq  i <  j \le  n \Rightarrow  A_j \cap  B_i \neq
    \emptyset )$.  Indeed, consider a smallest set $B_i$ such that $B_i
    \cap A_i =  \emptyset$ and such that for every  $A_j$ (with $j \leq
    n$) not included in $A_i$: $B_i \cap A_j \neq \emptyset$. Note that
    $B_i$ exists and that $1 \leq i  < j \le n \Rightarrow A_j \cap B_i
    \neq \emptyset$. So,  if $|B_i| = k$, we are  done.  If $|B_i|< k$,
    we are done easily by  completing $B_i$ with elements not in $A_i$.
    If $|B_i|  \ge k+1$, let $B_i  = \{e_1, \dots,  e_{k+1}, \dots \}$.
    By  minimality,  for  every  $h  \in [k+1]$,  there  exists  a  set
    $A_{i_h}$  such that $B_i  \cap A_{i_h}  \neq \emptyset$  and $(B_i
    \backslash  \{e_{i_h}\})  \cap A_{i_h}  =  \emptyset$.  Hence,  the
    incidence  matrix of  the sets  $A_i, A_{i_1},  \dots, A_{i_{k+1}}$
    over  the elements  $e_1, \dots,  e_{k+1}$ is  the  $k+1$ singleton
    matrix, contradicting  the fact that condition~1 does  not hold for
    $\cal H$.

For every set  $B_i = \overline{H_i}$, $i \geq n+1$,  we claim that we
    can construct  a set  $A_i$ such that  $B_i \cap A_i  = \emptyset$,
    $|A_i| = l$ and for every $j$, $(n < j < i \Rightarrow A_i \cap B_j
    \neq \emptyset )$.  If not, as in the preceeding paragraph, we find
    an  $(l+1)$-singleton  matrix in  the  collection  of the  $B_i$'s,
    $i\geq    n+1$.   Thus,   by    complementation,   we    find   the
    $l+1$-cosingleton matrix  in $\cal H$, contradicting  the fact that
    condition~2 does not hold for $\cal H$.

Finally,  we claim  that if  $m\geq i\geq  n+1> n  \geq j\geq  1$, then
    $A_i\cap B_j \neq  \emptyset$.  Suppose not, and let  us consider $m
    \geq i  \geq n+1$ and $n  \geq j \geq 1$  such that $A_i  \cap B_j =
    \emptyset$.   Since  $|A_i| =  l$  and $|B_j|  =  k$,  we know  that
    $(A_i,B_j)$ is a  partition of $[k+l]$. Since $A_j \cap B_j = $
$\emptyset=$ $A_i \cap B_i$, we have $A_j \subset A_i
\subset \overline{B_i}$.   Since
$H_j  =  A_j$  and
$H_i  =  \overline{B_i}$, we  obtain $H_j \subset H_i$,  contradicting the
 fact
    that condition~3 does not hold for $\cal H$.

Theorem~\ref{thfrankl},  and the sets $A_i$ and  $B_i$ imply $m\leq
     \comb{k+l}{l}$.

\end{proof}

Note  that in  the case  $k=l$,  Theorem~\ref{notth}  is not  an immediate
     consequence of Sperner's lemma, which  states that in any collection of
     ${2n \choose n}+1$ subsets of  $[2n]$ one can find a subset included
     in  another one  (see~\cite{lint:combinatorics:extremal}).  Indeed,
     Sperner's lemma says nothing about the size of the two subsets.

\appendix

\section*{Appendix: proof of $S(2,2)=8$}
\label{app}

We prove here that $S(2,2)=8$.  Our proof is long and requires several
   lemmas, some of  which may give ideas for  more general results. It
   will be convenient to work with \emph{reduced} collections of sets,
   in a sense that we define now.

\begin{defeng}
We say  that a collection $\mathcal{H}=\{H_1,H_2,\dots,  H_m\}$ of $m$
   distinct sets is reduced if  every element is \emph{useful} to make
   the sets distinct, that is  for every $x \in H_1\cup\dots\cup H_m$,
   we can  find $i$ and $j$ such  that $i \neq j$  and $H_i \backslash
   \{x\} = H_j \backslash \{x\}$.

\end{defeng}

Note  that  in a  reduced  collection of  sets,  there  cannot be  any
     \emph{universal} element, i.e., there  is no element in $H_1 \cap
     H_2  \cap \dots \cap  H_m$. Also  there are  no \emph{duplicated}
     elements, that  is for every  $x$ and $y$  in $H_1 \cup  H_2 \cup
     \dots \cup  H_m$, we  can find  $i$ and $j$  such that  $H_i \cap
     \{x,y\} \neq H_j \cap \{x,y\}$.

From a collection $\mathcal{H}$ of distinct sets, we can get a reduced
     collection  $\mathcal{H}'$ of  the same  cardinality  by deleting
     useless  elements  as  long  as  there  are  any  (the  resulting
     $\mathcal{H}'$ may depend on the choice of the arbitrary order of
     the deletion of the useless  elements).  We say that $\cal H'$ is
     \emph{obtained} from  $\cal H$.  If a  singleton, co-singleton or
     monotone matrix is found in  $\mathcal{H}'$, then it can be found
     in $\mathcal{H}$.   This is why, when computing  $S(k,l)$, we can
     suppose that the collections of sets we consider are reduced.

The following two  lemmas give answers to natural  questions: How many
     elements  are there  in a  reduced collection  of sets?   Given a
     reduced collection  of sets ${\cal H}  = \{ H_1,  H_2, \dots, H_m
     \}$, if  elements $e_1, e_2,  \dots, e_k$ are picked  in $H_1\cup
     H_2\cup \dots  \cup H_m$, how  many $H_i$'s have  distinct traces
     over  $\{e_1, \dots,  e_k\}$  ?  The  lemma  below is  implicitly
     stated in an article of Kogan~\cite{kogan:testset}.  That article
     gives an interesting characterization of the structure of special
     reduced collections of sets.

\begin{lemma}[\cite{kogan:testset}]  \label{lemmekogan}
Let $\mathcal{H} =  \{H_1, H_2, \dots, H_m\}$ be  a reduced collection
     of $m$  sets.  Then $H_1 \cup  \dots \cup H_m$ has  at most $m-1$
     elements.
\end{lemma}

\begin{proof}
Easy induction on $m$.
\end{proof}

\begin{lemma} \label{lemmereduit}
Let $\mathcal{H} =  \{H_1, H_2, \dots, H_m\}$ be  a reduced collection
     of sets.   If $e_1, e_2, \dots,  e_k$ are distinct  elements of $
     H_1 \cup H_2 \cup \dots \cup  H_m$ then we can find $k+1$ $H_i$'s
     that  are  distinct  over  $e_1,  e_2, \dots,  e_k$,  i.e.,  sets
     $H_{i_1},  H_{i_2},  \dots,   H_{i_{k+1}}$  such  that  the  sets
     $H_{i_1}  \cap \{e_1,  e_2, \dots,  e_k\}$, $H_{i_2}  \cap \{e_1,
     e_2, \dots,  e_k\}$, \dots, $H_{i_{k+1}} \cap  \{e_1, e_2, \dots,
     e_k\}$ are distinct.
\end{lemma}

\begin{proof}
Easy induction on $k$.
\end{proof}

From  now on,  for simplicity  we will  make no  difference  between a
    collection  of sets  and its  incidence  matrix, in  which we  can
    rearrange rows and columns. When a matrix is given, we call $r_1$,
    $r_2$, \dots  its rows and  $c_1$, $c_2$, \dots its  columns.  The
    incidence matrix of  a reduced collection of sets  is a 0-1 matrix
    where all  rows are  distinct, all columns  are distinct,  and for
    each column,  there exist  two rows that  become identical  if one
    erases the column. This implies that each column contains at least
    one 1 and  one 0. We use the notation: $r_i  \mapsto r_j [c_k]$ to
    express the  facts that  $r_i$ and $r_j$  are identical  except in
    column $c_k$,  and row $r_j$  (resp. $r_i$) has  a 1 (resp.  0) at
    column $c_k$.

We need seven more lemmas to show that $S(2,2)=8$.

\begin{lemma} \label{lemme94}
Let  $\cal  H$  be  a  collection  of nine  distinct  sets  over  four
     elements. We can find $S^3$, $\bar{S}^3$ or $M^3$ in $\cal H$.
\end{lemma}

\begin{proof}
Note that  $\cal H$ is necessarily  reduced since we  cannot have nine
   distinct sets  over three  elements.  First remark  that if  we have
   seven distinct sets over three  elements, then obviously, we find in
   them $S^3$ or $\bar{S}^3$.  Let  $\cal H$ be a reduced collection of
   nine distinct sets over four elements $c_1,c_2,c_3,c_4$.  Because of
   the  preceeding remark,  at most  six of  the rows  of $\cal  H$ are
   distinct  over $c_1, c_2,  c_3$.  At  least five  of those  rows are
   distinct over $c_1, c_2, c_3$  (if not, we cannot have nine distinct
   rows just with the column $c_4$).  Furthermore it is impossible that
   three rows of $\cal H$ are equal over $c_1, c_2, c_3$.

Hence, we  can find in $\cal H$  six rows $r_1, \dots,  r_6$ such that
     $r_1$,  $r_2$,  $r_3$  are   distinct  over  $c_1,  c_2,  c_3$,  and
     $r_1 \mapsto r_4 [c_4]$,     $r_2 \mapsto r_5 [c_4]$,       $r_3
\mapsto r_6 [c_4]$.    By
     Lemma~\ref{lemmekogan},  we can suppose  without loss  of generality
     that $r_1$,  $r_2$ and  $r_3$ are distinct  over $c_1,  c_2$.
Forget $c_3$.  Since
     there  are only four  possible sets  over $c_1,c_2$,  we have  only two
     cases   to  consider  (the    other two  are   equivalent  by
     complementation):

\begin{itemize}

\item We find in $\cal H$ the matrix:
    \[  \left(
\begin{array}{cccc}
000\\
010\\
110\\
001\\
011\\
111
\end{array}
\right) \]

Here,  we find $M^3$ in $\cal H$.

\item We find in $\cal H$ the matrix:
\[\left(
\begin{array}{cccc}
000\\
100\\
010\\
001\\
101\\
011
\end{array}
\right)
\]

Here, we find $S^3$ in $\cal H$.

\end{itemize}

\hbox{} \end{proof}

The following  lemma will be used extensively in the sequel.

\begin{lemma} \label{lemmeutile}
Let $\mathcal{H}$ be a reduced collection  of sets in which we can find
     the   matrix  $000   \choose  111$.   Then,  we   can   find  $S^3$,
     $\bar{S}^3$ or $M^3$ in $\mathcal{H}$.
\end{lemma}

\begin{proof}
By Lemma~\ref{lemmereduit} we can find in $\cal H$ a matrix $M$ with four
     distinct rows,  and among them $000$  and $111$. Suppose that
     $M$ is not $S^3$, $\bar{S}^3$ or $M^3$. Then there are   only two
     cases to consider (the other  cases are equivalent by permuting rows
     or columns or swapping $0$ and $1$):

\begin{itemize}
\item We find in $\mathcal{H}$ the matrix $M = \left( \begin{array}{c}
000\\ 100 \\ 010 \\ 111 \end{array} \right)$.

There is no possibility to add a fifth row, different from the four
above, to this 3-column matrix without finding $S^3$ or
$M^3$. If we  delete the element corresponding  to the last  column, then two
     sets  of  $\mathcal{H}$  must  become  equal. There  are  then  four
     subcases: we  can add to $M$ the  row 001 (and then  we find $S^3$), or
     101   ($\rightarrow   M^3$),   or 011   ($\rightarrow  M^3$)   or   110
     ($\rightarrow M^3$).

\item We find in $\mathcal{H}$ the matrix $M = \left( \begin{array}{c}
000\\ 100 \\ 011 \\ 111 \end{array} \right)$.

Again, there is no possibility to add a fifth row to $M$. If we
delete the element  corresponding to the second column, then
two
     sets  of  $\mathcal{H}$ must  become  equal.   There  are then  four
     subcases: we  can add  to $M$ the  row 010 ($\rightarrow  M^3$), or 110
     ($\rightarrow  M^3$), or 001 ($\rightarrow  M^3$) or  101 ($\rightarrow
     M^3$).

\end{itemize}
\hbox{}

\end{proof}

\begin{lemma} \label{lemmej3}

Let  $\cal H$  be a  reduced  collection of  sets over  at least  five
    elements. If we find in $\cal H$ the matrix $\left( \begin{array}{c}
    0000 \\  1000 \\ 0100 \end{array}  \right)$ then we  can find $S^3$,
    $\bar{S}^3$ or $M^3$ in $\cal H$.

\end{lemma}

\begin{proof}
Assume that  we can find  in $\cal H$  this matrix and that  we cannot
    find $S^3$, $\bar{S}^3$  or $M^3$ in $\cal H$.  There exist $i$ and
    $j$ such that $r_i \mapsto r_j [c_3]$. $i\neq 1$, otherwise we find
    $S^3$.   By  Lemma~\ref{lemmeutile}, we  can  assume w.l.o.g.  that
    $i=2$, and  let $j=4$.  Similarly,  there exist $i'$ and  $j'$ s.t.
    $r_{i'}  \mapsto r_{j'}  [c_4]$, and  since we  assumed  that $S^3$
    cannot be  found in $\cal  H$ we can  put w.l.o.g.  $i'=3$  and let
    $j'=5$.  Thus, we have found in $\cal H$ the matrix $M$:

\[
M= \left(
\begin{array}{c}
0000\\
1000\\
0100\\
1010\\
0101
\end{array}
\right)
\]

There is by assumption a fifth element (or column) $c_5$. We consider
three cases (the number in a square will always represent
the current hypothesis):

\begin{description}

\item[First case:]
We find in $\cal H$:

\[
\begin{array}{cc}

\left(
\begin{array}{cc}
0000 & \fbox{0}\\
1000 & 0^2\\
0100 & 0^3\\
1010 & 0^1\\
0101 & 0^1\\
        & 1^4
\end{array}
\right)

&

\begin{array}{l}
^1 \text{by Lemma \ref{lemmeutile}} \\
^2 \text{because  $r_2 \mapsto r_4 [c_3]$} \\
^3 \text{ because $ r_3 \mapsto r_5 [c_4]$} \\
^4 \text{each column must have a 1 somewhere}
\end{array}
\end{array}
\]

By Lemma~\ref{lemmeutile}, we are allowed to put 1   only twice in the
last row. In all cases, we find  $S^3$.

\item[Second case:]
We find in $\cal H$:
\[
\begin{array}{cl}
     \left(
\begin{array}{cc}
0000 & \fbox{1}\\
1000 & \fbox{0}\\
0100 & 0^3  \\
1010 & 0^2  \\
0101 & 0^1
\end{array}
\right)
&
\begin{array}{l}
    ^1 \text{by Lemma \ref{lemmeutile} } \\
    ^2 \text{because  $r_2 \mapsto r_4 [c_3]$ }\\
    ^3 \text{ because $ r_3 \mapsto r_5 [c_4]$ }
\end{array}
\end{array}
\]

We find $S^3$.

\item[Third case:]
We find in $\cal H$:
\[
\begin{array}{cl}
     \left(
\begin{array}{cc}
0000 & \fbox{1}\\
1000 & \fbox{1}\\
0100 & 1^1 \\
1010 &  1^2\\
0101 &  1^3\\
        & 0^4
\end{array}
\right)
&
\begin{array}{l}
    ^1 \text{because of the second case by symmetry} \\
    ^2 \text{because $ r_2 \mapsto r_4 [c_3]$ }\\
    ^3 \text{because  $r_3 \mapsto r_5 [c_4] $} \\
    ^4 \text{let $r_6$ be such that $r_6 \mapsto r_i [c_5]$ for some $i$}
\end{array}
\end{array}
\]

By Lemma~\ref{lemmeutile}, $r_6$ cannot have three or four 1's on the
    first four  columns. If $r_6$  has at most  one 1 on  these columns,
    then  necessarily  we   are  done  by  Lemma~\ref{lemmeutile}.  Thus
    necessarily $r_6$ has exactly two 1's on the first four columns.

One of these 1  has to be in the first or  second column, otherwise we
    find $S^3$. By symmetry of the first two columns, we can assume that
    $r_6$ has a  1 in column $c_1$.  If the other 1 is  in column $c_4$,
    again  we find  $S^3$. If  it is  in column  $c_3$, we  are  done by
    Lemma~\ref{lemmeutile}. Thus we are  left with only one possibility,
    and we find:

\[
\begin{array}{cl}
\left(
\begin{array}{cl}
0000 & 1  \\
1000 & 1  \\
0100 & 1  \\
1010 & 1 \\
0101 & 1  \\
1100 & 0 \\
1100 & 1^1 \\
\end{array}
\right)
&
\begin{array}{l}
^1 \text{there exists $i$ s.t. $r_6 \mapsto r_i [c_5]$ }
\end{array}
\end{array}
\]

We find $\bar{S}^3$.

\end{description}

\hbox{}
\end{proof}

\begin{lemma} \label{lemmej4}

Let $\cal H$ be a reduced  collection of   sets over at least five
     elements. If we find in $\cal H$ the matrix $\left( \begin{array}{c}
     0000 \\  1000 \\ 0110 \end{array}  \right)$ then we  can find $S^3$,
     $\bar{S}^3$ or $M^3$ in $\cal H$.

\end{lemma}

\begin{proof}
By Lemma~\ref{lemmereduit},  we can  suppose that $r_1$,  \dots, $r_5$
    are distinct over  $c_1$, \dots, $c_4$. Assume that  we cannot find
    $S^3$, $\bar{S}^3$ or $M^3$  in $\cal H$.  By Lemma~\ref{lemmeutile}
    and Lemma~\ref{lemmej3},  $r_4$ and $r_5$ have both  exactly two 1's
    over $c_1$,  \dots, $c_4$.   For the rows  $r_4$ and $r_5$,  0011 or
    0101 bring $S^3$ and 0110  is impossible because of $r_3$. Thus, for
    $r_4$ and $r_5$ the only possibilities are 1001, 1010 and 1100.

Since erasing $c_2$ make two rows  equal, we can suppose that $r_4$ is
    $1100$ over the first four columns. And since erasing $c_3$ make two
    rows equal, we can suppose that  $r_5$ is $1010$ over the first four
    columns.  As each column must have at least one 1, we find:

\[
\left(
\begin{array}{r}
0000 \\
1000 \\
0110 \\
1100 \\
1010 \\
      1
\end{array}
\right)
\]

The last line has exactly one 1 in the first three columns by
Lemma~\ref{lemmeutile} and Lemma~\ref{lemmej3}. In each case, we
find $S^3$.

\end{proof}

\begin{lemma} \label{lemmej5}

Let $\cal  H$ be a reduced  collection of   sets  over at least
     five  elements.  If   we  find  in  $\cal  H$   the  matrix  $\left(
     \begin{array}{c} 0000 \\ 1000  \end{array} \right)$ then we can find
     $S^3$, $\bar{S}^3$ or $M^3$ in $\cal H$.

\end{lemma}

\begin{proof}
By Lemma~\ref{lemmereduit}, we can  suppose that $r_1, \dots, r_5$ are
distinct    over   $c_1,    \dots,   c_4$.  Assume that we cannot find $S^3$,
     $\bar{S}^3$ or $M^3$ in $\cal H$.   By   Lemma~\ref{lemmeutile},
~\ref{lemmej3}
and~\ref{lemmej4},  we know  that  $r_3$, $r_4$  and  $r_5$ must  have
exactly two 1 over $c_1$,...,$c_4$, and one of them  on $c_1$.
Finally, we find in $\cal
H$ the following matrix and then $S^3$.

\[
\left(
\begin{array}{c}
0000\\
1000\\
1100\\
1010\\
1001
\end{array}
\right)
\]
\end{proof}

\begin{lemma} \label{lemmej6}

Let $\cal  H$ be a reduced  collection of distinct sets  over at least
     five elements.  If we find in $\cal  H$ the matrix $  (0000)$ or the
     matrix $(1111)$,  then we can find $S^3$,  $\bar{S}^3$ or $M^3$
     in $\cal H$.

\end{lemma}

\begin{proof}
Let us suppose  that we find $(0000)$ in $\cal H$,  and assume that we
    cannot  find   $S^3$,  $\bar{S}^3$  or   $M^3$  in  $\cal   H$.   By
    Lemma~\ref{lemmej5}  we know  that  we cannot  have $r_1\mapsto  r_i
    [c_1]$  for some  $i$. Thus,  we  may suppose  that $r_2\mapsto  r_3
    [c_1]$ and we find in $\cal H$ the following matrix:

\[
\left(
\begin{array}{l}
0000\\
0\\
1
\end{array}
\right)
\]

On the columns $2,  3, 4$, if we complete the rows 2  and 3 by at most
    one 1 we  are done by Lemma~\ref{lemmej5}. If we  complete by two or
    three 1's, we are done  by Lemma~\ref{lemmeutile}. The case where we
    find $(1111)$ is similar by complementation.
\end{proof}

\begin{lemma} \label{lemmej8}

Let $\cal H$ be a reduced  collection of   sets over at least five
     elements.  We can find $S^3$, $\bar{S}^3$ or $M^3$ in $\cal H$.

\end{lemma}

\begin{proof}
Assume that we cannot find $S^3$, $\bar{S}^3$ or $M^3$ in $\cal H$. We
    suppose w.l.o.g.  that $r_1 \mapsto r_2  [c_1]$. If $\cal  H$ has at
    least  six columns,  then  we are  be  done by  Lemma~\ref{lemmej6}.
    Hence there  are exactly five columns,  and we can find  in $\cal H$
    the matrix:

\[
\left(
\begin{array}{l}
00011\\
10011
\end{array}
\right)
\]

(1) We first suppose that $\cal H$ has at most nine sets. If we delete
    any one of the columns $c_1, \dots, c_5$, then two rows must become
    equal. Since  there are at  most nine rows,  we know that a  row is
    involved twice in this process, say w.l.o.g. $r_1$ or $r_2$. We can
    assume it  is $r_1$ by  symmetry.  If we  can find $r_3$  such that
    $r_3  \mapsto  r_1  [c_j]$, with  $j=4$  or  $5$,  we are  done  by
    Lemma~\ref{lemmej6}. Hence we can assume w.l.o.g. that $r_1 \mapsto
    r_3 [c_2]$.  We find in $\cal H$ the matrix:

\[
\left(
\begin{array}{rl}
00011  &     \\
10011  &     \\
01011  &
\end{array}
\right)
\]

Let $r_4$  be such that $r_i \mapsto  r_4 [c_3]$ for some  $i$. By the
    argument  of the beginning  of this  proof, we  know that  $r_4$ has
    exactly three  1's and  two 0's.  Exactly one of  these 1's  is over
    $c_1$,  $c_2$ (by  Lemma~\ref{lemmeutile}  and to  avoid $S^3$).  By
    symmetry between  $c_1$ and $c_2$,  and between $c_4$ and  $c_5$, we
    obtain w.l.o.g. the following matrix:

\[
\begin{array}{cl}
\left(
\begin{array}{rl}
00011  &     \\
10011  &     \\
01011  &     \\
10101  &    \\
10001  &  ^1 \\
       0  &  ^2
\end{array}
\right)
&
\begin{array}{l}
^1 \text{take $i=5$ in $r_i \mapsto r_4 [c_3]$ }\\
^2 \text{each column contains a 0}\\
\end{array}
\end{array}
\]

We now consider two cases:
\begin{description}
\item[First case:]
We find in $\cal H$:

\[
\begin{array}{cl}
\left(
\begin{array}{ccccc}
0&0&0&1&1\\
1&0&0&1&1\\
0&1&0&1&1\\
1&0&1&0&1\\
1&0&0&0&1\\
\fbox{1}&1^2&0^3&0^1&0
\end{array}
\right)
&
\begin{array}{l}
^1 \text{if 1, we find $\bar{S}^3$ with the columns 1, 4, 5.}\\
^2 \text{if 0, we are done by Lemma \ref{lemmeutile}.}\\
^3 \text{if 1, we are done by Lemma \ref{lemmeutile}.}
\end{array}
\end{array}
\]

We find $S^3$ (columns 2, 3 and 4).

\item[Second case:]

We find in $\cal H$:
\[
\begin{array}{cl}
\left(
\begin{array}{ccccc}
0&0&0&1&1\\
1&0&0&1&1\\
0&1&0&1&1\\
1&0&1&0&1\\
1&0&0&0&1\\
\fbox{0}&0^2&1^1&1^1&0
\end{array}
\right)
&
\begin{array}{l}
^1 \text{by Lemma \ref{lemmeutile}.}\\
^2 \text{by Lemma \ref{lemmeutile}.}
\end{array}
\end{array}
\]

We find $S^3$ (columns 1, 2 and 3).

\end{description}

(2) So the  lemma is proved  unless $\cal H$  has more than nine  sets. In
     this case, we  pick nine of them. If they  form a reduced collection
     then we are done. If they  do not, we delete a  useless element.
    We stay with nine distinct sets defined over four elements,
and   we are done by Lemma~\ref{lemme94}.

\end{proof}

Now, by Lemma~\ref{lemme94} and Lemma~\ref{lemmej8} we obtain:

\begin{proposition} \label{profinal}
Let $\cal H$ be  a collection of at least 9 distinct  sets. We can find
     $S^3$, $\bar{S}^3$ or $M^3$ in $\cal H$. Hence $S(2,2)=8$.
\end{proposition}

\input{fichier.bbl}


\end{document}

%% file: fichier.bbl

%% file: submited.20030909.bbl
\begin{thebibliography}{1}

\bibitem{bauslaugh:list}
Bruce~L. Bauslaugh.
\newblock List-compactness of infinite directed graphs.
\newblock {\em Graphs and Combinatorics}, 17:17--38, 2001.

\bibitem{frankl:two}
P.~Frankl.
\newblock An extremal problem for two families of sets.
\newblock {\em European Journal of Combinatorics}, 3:125--127, 1982.

\bibitem{furedi:hypergraphs}
Z.~F{\"{u}}redi and Zs. Tuza.
\newblock Hypergraphs without a large star.
\newblock {\em Discrete Mathematics}, 55:317--321, 1985.

\bibitem{grh:ramseytheory}
R.~L. Graham, B.~L. Rothschild, and J.~H. Spencer.
\newblock {\em Ramsey Theory}.
\newblock John Wiley and sons, New York, 1980.

\bibitem{kalai:intersection}
G.~Kalai.
\newblock Intersection patterns of convex sets.
\newblock {\em Israel J. Math.}, 48:161--174, 1984.

\bibitem{kogan:testset}
A.~Kogan.
\newblock On the essential test sets of discrete matrices.
\newblock {\em Discrete Applied Mathematics}, 60:249--255, 1995.

\bibitem{nicolas:dea}
N.~Trotignon.
\newblock L'argument de compacit{\'e}, 2001.
\newblock M{\'e}moire de DEA (in french).

\bibitem{lint:combinatorics:extremal}
J.~H. van Lint and R.~M. Wilson.
\newblock {\em A Course in Combinatorics}, Chapter~6.
\newblock Cambridge University Press, 1992.

\end{thebibliography}
